\documentclass[12pt, a4paper, reqno]{amsart}
\usepackage[utf8]{inputenc}
\usepackage{color}
\usepackage{mathrsfs}
\usepackage{graphicx}
\usepackage{faktor}
\usepackage{amssymb}
\usepackage{amsmath, float}
\usepackage{amsthm, csquotes, microtype, fullpage, fbb, textcomp, wrapfig}
\usepackage[colorlinks=true, linkcolor=blue, citecolor=blue, urlcolor=blue, breaklinks=true]{hyperref}
\usepackage[capitalise]{cleveref}
\usepackage[all,cmtip]{xy}
\usepackage{float}
\usepackage{tikz}
\usetikzlibrary{arrows,shapes,automata,backgrounds,petri,decorations}
\usetikzlibrary{snakes}
\usetikzlibrary{automata} 
\usetikzlibrary[automata] 
\usepackage{geometry}
\usepackage{pdflscape}
\usetikzlibrary{external,automata,trees,positioning,shadows,arrows,shapes.geometric}
\usepackage{tikz-cd}
\usetikzlibrary{decorations.pathreplacing,decorations.markings}
\tikzset{close/.style={near start,outer sep=-2pt}} 
\tikzset{
  on each segment/.style={
    decorate,
    decoration={
      show path construction,
      moveto code={},
      lineto code={
        \path [#1]
        (\tikzinputsegmentfirst) -- (\tikzinputsegmentlast);
      },
      curveto code={
        \path [#1] (\tikzinputsegmentfirst)
        .. controls
        (\tikzinputsegmentsupporta) and (\tikzinputsegmentsupportb)
        ..
        (\tikzinputsegmentlast);
      },
      closepath code={
        \path [#1]
        (\tikzinputsegmentfirst) -- (\tikzinputsegmentlast);
      },
    },
  },
  mid arrow/.style={postaction={decorate,decoration={
        markings,
        mark=at position .5 with {\arrow[#1]{stealth}}
      }}},
}

\usepackage{mathtools}
\usepackage{hyperref,url}

\newtheorem{thm}{Theorem}[section]

\newtheorem{lem}[thm]{Lemma}
\newtheorem{prop}[thm]{Proposition}

\theoremstyle{definition}
\newtheorem{defn}[thm]{Definition}
\newtheorem{exam}[thm]{Example}
\newtheorem{que}[thm]{Question}
\theoremstyle{definition}

\newtheorem{obs}[thm]{Observation}
\newtheorem{nt}[thm]{Note}
\numberwithin{equation}{section}

\newcommand{\ZZ}{\mathbb{Z}}

\newcommand{\Z}{\mathbb{Z}}

\begin{document}

\title{On the structure of finitely presented Bestvina-Brady groups}

\author{Priyavrat Despande}
\address{Chennai Mathematical Institute, India.}
\email{pdeshpande@cmi.ac.in}

\author{Mallika Roy}
\address{Departament of Mathematics, Universidad del Pa\`is Vasco (UPV/EHU), Spain.}
\email{mallika.roy@upc.edu, mallikaroy75@gmail.com}

\begin{abstract}
Right-angled Artin groups and their subgroups are of great interest because of their geometric, combinatorial and algorithmic properties. 
It is convenient to define these groups using finite simplicial graphs. 
The isomorphism type of the group is uniquely determined by the graph. 
Moreover, many structural properties of right angled Artin groups can be expressed in terms of their defining graph. 

In this article we address the question of understanding the structure of a class of subgroups of right-angled Artin groups in terms of the graph. 
Bestvina and Brady, in their seminal work, studied these subgroups (now called Bestvina-Brady groups or Artin kernels) from a finiteness conditions viewpoint. 
Unlike the right-angled Artin groups the isomorphism type of Bestvina-Brady groups is not uniquely determined by the defining graph. 
We prove that certain finitely presented Bestvina-Brady groups can be expressed as an iterated amalgamated product. 
Moreover, we show that this amalgamated product can be read off from the graph defining the ambient right-angled Artin group. 
\end{abstract}
\keywords{right angled Artin group, amalgamated product, Bestvina-Brady group}
\subjclass[2010]{20F36, 20F65, 08B25}
\maketitle

\section{Introduction}\label{intro}

A right-angled Artin group (a RAAG, for short) is a finitely presented group such that the commuting relations are the only relations. 
It is perhaps easier to describe this group using finite simplicial graphs. 
Let $\Gamma$ be such a graph; then the associated RAAG, denoted by $A_{\Gamma}$, has generators corresponding to vertices of $\Gamma$ and two generators commute whenever the corresponding vertices are connected by an edge.
We refer the reader to \cite{Charney} for an encyclopedic introduction to RAAGs. 
They have become central in group theory, their study interweaves geometric group theory with other areas of mathematics. This class interpolates between two of the most classical families of groups, free and free abelian groups, and its study provides uniform approaches and proofs, as well as rich generalisations of the results for free and free abelian groups. The study of this class from different perspectives has contributed to the development of new, rich theories such as the theory of CAT(0) cube complexes and has been an essential ingredient in Ian Agol's solution to Thurston’s virtual fibering and virtual Haken conjectures.
RAAGs are important in geometric group theory for many reasons, including the fact that they have interesting subgroups; for example, Bestvina-Brady groups. 

A Bestvina-Brady group (a B-B group, for short) is the kernel of the group homomorphism $A_\Gamma \rightarrow \ZZ$ which takes all the generators of $A_\Gamma$ to $1$. 
One of the reasons they are interesting is that they provide an example of a group that satisfies the finiteness property $\mathbf{FP_n}$ but not $\mathbf{FP_{n+1}}$. 
They also provide counterexamples either to the Eilenberg–Ganea Conjecture or the Whitehead Conjecture; see \cite{BB} for details.

Though RAAGs interpolate between free groups and free abelian groups their structure is not always straightforward. 
However, many of their structural properties can be read off from the underlying graph. 
For example, a RAAG $A_{\Gamma}$ is a free product of two of its sub-RAAGs if and only if $\Gamma$ is disconnected; by a theorem of Clay \cite{Clay} all nontrivial splittings of $A_{\Gamma}$ over $\Z$ correspond to cut vertices of $\Gamma$, further it was proved by Groves and Hull \cite{GH} that $A_{\Gamma}$ splits over an abelian group if and only if $\Gamma$ is disconnected, or complete or contains a separating clique.

It is easy to observe that two non-isomorphic graphs can give rise to the same B-B group $H_{\Gamma}$ (see~\ref{imp note on non-isomorphic}).
Moreover, though the finiteness properties are completely determined by the (topology of the) clique complex of $\Gamma$, not much work has been done to understand structural properties of $H_{\Gamma}$ in terms of graph theoretic input. 
We should mention here the recent work of Chang \cite{Chang-splitting} on abelian splittings of B-B groups and the work of Barquinero-Ruffoni-Ye \cite{BRY} decomposing Artin kernels (a class of subgroups generalizing B-B groups) as graphs groups. 
In both the works the decomposition of $H_{\Gamma}$ is expressed in terms of the underlying $\Gamma$; hence we consider them as a motivation for our paper. 

The aim of this article is to show that a certain class of finitely presented B-B groups can be decomposed as an (iterated) amalgamated product of a RAAG and finitely many copies of $\Z^2$. 
We express this amalgamated product completely in terms of the underlying graph. 
To be precise, we decompose $\Gamma$ as a union of $\Gamma'$ and some triangles (i.e., $3$-cliques); this union is with the help of a suitable spanning tree. 
The subgraph $\Gamma'$ is selected on the basis that the corresponding B-B group $H_{\Gamma'}$ is isomorphic to a RAAG; the triangles correspond to copies of $\ZZ^2$. 
The classes of graphs for which this decomposition works include the $1$-skeleta of certain (extra)-special triangulations of the $2$-disk and connected graphs with a separating clique $K_n, n \geq 3$. 
Our main theorem implies that such B-B groups can be decomposed as an iterated amalgamation of RAAGs. 
We should note here that Papadima and Suciu \cite[Proposition 9.4]{PS} showed that the B-B group corresponding to an extra-special triangulation is not isomorphic to any RAAG. 
The existence of a separating clique also implies that the corresponding B-B group splits over an abelian subgroup, see Chang \cite[Theorem 3.9]{Chang-splitting}; we expand on this aspect in \Cref{conrem}. 

\section{Preliminaries and notations}\label{prelims}

Below we present the necessary preliminaries on graph theory and B-B groups.

\subsection{Graph theory} We recall and set up some basic notations and terminologies in graph theory. 
Throughout this article, we assume finite graphs which have no loops and multi-edges, i.e., all the graphs are finite simplicial. 
Given a graph $\Gamma$, we denote the set of its vertices and edges by $V(\Gamma)$ and $E(\Gamma)$, respectively. 
We denote $e=(v, w)$ to be an edge connecting vertices $v$ and $w$. 
The initial and terminal vertices of the edge $e$ are denoted by $\iota(e)$ and $\tau(e)$, respectively, such that $e=(\iota(e), \tau(e))$.
Two vertices are adjacent if they are connected by an edge. 
A \textit{spanning tree} of $\Gamma$ is a subgraph of $\Gamma$ which is a tree and contains every vertex of $\Gamma$.
As we are dealing with finite graphs it is straightforward to see that, for any spanning tree $T$ of $\Gamma$, $|E(T)| = |V(\Gamma)| - 1$.
Given any subset $V'$ of $V(\Gamma)$, the \textit{induced subgraph} (in some literature, it is also called as full subgraph) on $V'$ is a graph $\Gamma'$ whose vertex set is $V'$, and two vertices are adjacent in $\Gamma'$ if and only if they are adjacent in $\Gamma$. 

The \textit{star graph} $S_n$ of order $n$, sometimes simply known as an \textit{n-star}, is a tree on $n$ vertices with one vertex having degree $n-1$ and the other $n-1$ vertices having degree~$1$. 

Given two graphs $\Gamma_1 = (V (\Gamma_1), E(\Gamma_1))$ and $\Gamma_2 = (V (\Gamma_2), E(\Gamma_2))$, the \textit{union} of
$\Gamma_1$ and $\Gamma_2$ is $\Gamma_1 \cup \Gamma_2 = (V (\Gamma_1) \cup V (\Gamma_2), E(\Gamma_1) \cup E(\Gamma_2))$. We denote the \textit{disjoint union} of two graphs $\Gamma_1$ and $\Gamma_2$ as $\Gamma_1 \sqcup \Gamma_2$, i.e., $\Gamma_1$ and $\Gamma_2$ share no vertices.
The \textit{join} of two graphs $\Gamma_1$ and $\Gamma_2$, denoted by $\Gamma_1 \vee \Gamma_2$, is defined to be the graph union $\Gamma_1 \cup \Gamma_2$ and with every pair of
vertices $(v, w) \in V (\Gamma_1) \times V (\Gamma_2)$ being adjacent.
The join operation is commutative, that is, $\Gamma_1 \vee \Gamma_2= \Gamma_2 \vee \Gamma_1$. When a graph $\Gamma$ decomposes as a join of a vertex $v$ and another graph $\Gamma'$, the vertex $v$ is called a 
\textit{dominating vertex}, and $\Gamma$ is called a \textit{cone graph} or the \textit{cone on $\Gamma'$}. 

Two graphs $\Gamma_1$ and $\Gamma_2$ are said to be isomorphic, denoted by $\Gamma_1 \cong \Gamma_2$, if there is a bijection $\varphi : V(\Gamma_1) \rightarrow V(\Gamma_2)$ such that two vertices $v, w$ are adjacent in $\Gamma_1$ if and only if $\varphi(v), \varphi(w)$ are adjacent in $\Gamma_2$. 
The star graph $S_n$ is isomorphic to the complete bipartite graph $K_{(1,n-1)}$. Also let $\Gamma_1$, $\Gamma_2$ be cone graphs on $\Gamma'_1 \text{ and } \Gamma'_2$ respectively, then $\Gamma_1 \cong \Gamma_2$ if and only if $\Gamma'_1 \cong \Gamma'_2$.

Given a graph $\Gamma$ we construct a simplicial complex $\triangle_\Gamma$, called the \emph{flag complex}, as follows: 
the vertex set is the ground vertex set $V(\Gamma)$ and a subset of cardinality $k$ is a $(k-1)$-simplex if and only if the induced subgraph is a $k$-clique. 
In the literature, the term clique complex is also used for the flag complex.
Note that, we do not differentiate between an abstract simplicial complex and its geometric realization; any topological statement about the flag complex (equivalently, the clique complex) is about its geometric realization. 
Our main focus is on those graphs whose flag complex is simply connected. 
Such graphs form a fairly large class, for example, a connected chordal graph has the contractible flag complex.

\subsection{Bestvina--Brady groups}\label{bb_intro}

\begin{defn}
Let $\Gamma$ be a finite simplicial graph with the vertex set $V(\Gamma)$ and the edge set $E(\Gamma)$. 
The right-angled Artin group $A_{\Gamma}$ associated to $\Gamma$ has the following finite presentation:
\[
A_\Gamma = \bigl\langle V(\Gamma) \mid [v,w]=1 \text{ for each edge } (v,w) \in E(\Gamma) \bigr\rangle.
\]
Let $\varphi \colon A_\Gamma \rightarrow \ZZ$ be the group homomorphism sending all generators of $A_\Gamma$ to $1$. The \textit{Bestvina–Brady group} $H_\Gamma$ associated to $\Gamma$ is the kernel, $\ker \varphi$.
\end{defn}


We already mentioned in~\Cref{intro} that the B-B groups were first introduced in the influential work of Bestvina and Brady~\cite{BB} as an answer to a long standing open question regarding the existence of non-finitely presented groups of type \textbf{FP}---a result based on homological group theory. 

\vspace{0.4 cm}
\begin{thm}[\cite{BB}, Main Theorem]
Let $\Gamma$ be a finite simplicial graph.
\begin{itemize}
    \item[(1)] $H_\Gamma$ is finitely generated if and only if $\Gamma$ is connected.
    \item[(2)] $H_\Gamma$ is finitely presented if and only if $\triangle_\Gamma$ is simply-connected.
    \item [(3)] $H_\Gamma$ is of type $\mathbf{FP_{n+1}}$ if and only if $\Gamma$ is $n$-acyclic.
\end{itemize}
\end{thm}

This result includes the Stallings' group~\cite{S} --- $H_\Gamma$ associated to the RAAG $F_2 \times F_2 \times F_2$ --- an example of finitely presented but not of type $\mathbf{FP_3}$ and R. Bieri's group~\cite{B} of type $\mathbf{FP_{n}}$ but not of type $\mathbf{FP_{n+1}}$, which is $H_\Gamma$ corresponding to the $\Gamma$, a join of $(n+1)$ pairs of points. 

The presentation of B-B groups was described by Dicks--Leary in~\cite{DL}:

\begin{thm}[\cite{DL}, Theorem 1]
 Let $\Gamma$ be connected. The group $H_\Gamma$ has a presentation with generators the set of directed edges of $\Gamma$, and relators all words of the form $e^n_1 e^n_2 \cdots e^n_{\ell}$, where $\ell,n \in \ZZ, n \geq 0,\ell \geq 2$, and $(e_1, \ldots,e_\ell)$ is a directed cycle in $\Gamma$.  In terms of the given generators for $A_\Gamma, e=\iota e(\tau e)^{-1}$. 
\end{thm}

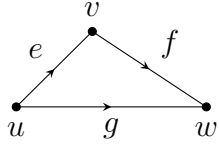
\begin{figure}[H]
    \centering
    \begin{tikzpicture}
    \tikzset{
    edge/.style={draw=black,postaction={on each segment={mid arrow=black}}}
} 
\node[fill=black!100, state, scale=0.10, vrtx/.style args = {#1/#2}{label=#1:#2}] (A) [vrtx=below/$u$]     at (0, 0) {};
\node[fill=black!100, state, scale=0.10, vrtx/.style args = {#1/#2}{label=#1:#2}] (C) [vrtx=above/$v$]    at (1, 1) {};
\node[fill=black!100, state, scale=0.10, vrtx/.style args = {#1/#2}{label=#1:#2}] (B) [vrtx=below/$w$]     at (2.5, 0) {};

\draw[edge] (A) -- (B) node[midway, below] {$g$};
\draw[edge] (C) -- (B) node[midway, above right] {$f$};
\draw[edge] (A) -- (C) node[midway, above left] {$e$};
    \end{tikzpicture}
    \caption{A directed triangle. 
    }
    \label{directed triangle}
\end{figure}

Let us fix a linear order on the vertices, and orient the edges increasingly. A triplet of
edges $(e, f, g)$ forms a directed triangle  if $e = (u, v), f = (v, w), g = (u, w)$, and
$u < v < w$; see Figure~\ref{directed triangle}.


When $\Gamma$ is connected and $\triangle_\Gamma$ is simply connected, we can write down a presentation for $H_\Gamma$, called the \textit{Dicks--Leary
presentation} \cite[Theorem 1, Corollary 3]{DL}. In other words, the aforementioned theorem could be simplified as:

\begin{thm}
Suppose the flag complex $\triangle_\Gamma$ is simply connected. Then $H_\Gamma$ has presentation
\begin{equation*}
H_\Gamma = \langle e \in E(\Gamma) \mid ef = fe, ef = g \text{ if }\triangle(e, f, g) \text{ is a directed triangle } \rangle.    
\end{equation*}
Moreover, the inclusion $\iota \colon H_\Gamma \hookrightarrow A_\Gamma$ is given by $\iota({e}) = uv^{-1}$ for every edge ${e} = (u, v)$ of $\Gamma$.
\end{thm}



The Dicks-Leary presentation is not necessarily a minimal presentation, i.e., there are some redundant generators. Dicks-Leary considered all the edges of $\Gamma$ as the generators.
The simpler presentation was given by Papadima-Suciu in ~\cite{PS}. The authors proved that for the generators of $H_\Gamma$ it is enough to consider the edges of a spanning tree of $\Gamma$.

\begin{thm}[{\cite[Corollary 2.3]{PS}}]\label{Suciu-Papadima presentation}
If $\triangle_\Gamma$ is simply-connected, then $H_\Gamma$ has a presentation $H_\Gamma = F/R$, where $F$ is the free group generated by the edges of a spanning tree of $\Gamma$, and $R$ is a finitely generated normal subgroup of the commutator group $[F, F]$.
\end{thm}

Here are some examples of B-B groups. 

\begin{exam}
If $\Gamma$ is a complete graph on $n$ vertices, then any spanning tree has $n-1$ edges. In fact, we can choose the spanning tree as a star graph. Moreover, any two edges in the spanning tree form the two sides of a triangle in $\Gamma$. 
The corresponding Bestvina--Brady group has $n-1$ generators and any two of them commute, hence it is $\ZZ^{n-1}$. 
\end{exam}

\begin{exam}\label{free gp BB}
Now let $\Gamma$ be a tree on $n$ vertices.
The spanning tree is the graph $\Gamma$ itself and there are no triangles. 
The corresponding Bestvina--Brady group has $n-1$ generators and none of them commute, hence it is $F_{n-1}$ the free group on $n-1$ generators.
\end{exam}

\begin{exam}
Let $\Gamma$ be the graph in \Cref{fig:exam}. 
Choosing the spanning tree $T = \{e_1 , \ldots , e_5\}$ as indicated, the presentation of the B-B group reads as follows:
\[
H_\Gamma = \langle e_1,  \ldots, e_5 \mid [e_1 , e_2], [e_2 , e_3], [e_3 , e_4], e_5{e_2}^{-1}e_3 = {e_2}^{-1}e_3 e_5 \rangle.
\]
Unlike the previous two examples, this particular $H_\Gamma$ is not isomorphic to any RAAG (see \cite[Proposition 9.4]{PS} for details). 
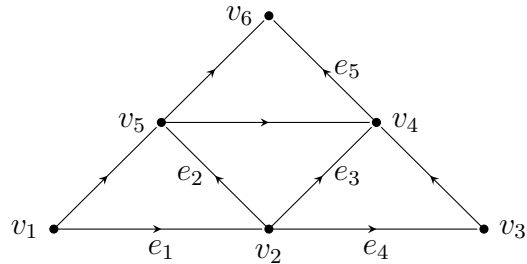
\begin{figure}[H]
    \centering
\begin{tikzpicture}[shorten >=1pt,node distance=20cm,auto]
\tikzset{
    edge/.style={draw=black,postaction={on each segment={mid arrow=black}}}
}
\node[fill=black!100, state, scale=0.10, vrtx/.style args = {#1/#2}{label=#1:#2}] (1) [vrtx=left/$v_6$] {};

\node[fill=black!100, state, scale=0.10, vrtx/.style args = {#1/#2}{label=#1:#2}] (2) [vrtx=right/$v_4$] [ below right of = 1] {};

\node[fill=black!100, state, scale=0.10, vrtx/.style args = {#1/#2}{label=#1:#2}] (3) [vrtx=left/$v_5$] [ below left of = 1] {};

\node[fill=black!100, state, scale=0.10, vrtx/.style args = {#1/#2}{label=#1:#2}] (4) [vrtx=right/$v_3$] [ below right of = 2] {};

\node[fill=black!100, state, scale=0.10, vrtx/.style args = {#1/#2}{label=#1:#2}] (5) [vrtx=below/$v_2$] [ below left of = 2] {};

\node[fill=black!100, state, scale=0.10, vrtx/.style args = {#1/#2}{label=#1:#2}] (6) [vrtx=left/$v_1$] [ below left of = 3] {};


\draw[edge] (6) -- (5) node[midway, below] {$e_1$};
\draw[edge] (5) -- (3) node[midway, left] {$e_2$};
\draw[edge] (5) -- (2) node[midway, right] {$e_3$};
\draw[edge] (5) -- (4) node[midway, below] {$e_4$};
\draw[edge] (2) -- (1) node[midway, right] {$e_5$};
\draw[edge] (6) -- (3);
\draw[edge] (3) -- (2);
\draw[edge] (3) -- (1);
\draw[edge] (4) -- (2);
\end{tikzpicture}
\caption{A graph whose corresponding B-B group is \textit{not} a RAAG.}
\label{fig:exam}
\end{figure}
\end{exam}

Here we also recall the definition of an amalgamated product. 

\begin{defn}
Let $G_1$ and $G_2$ be groups with distinguished isomorphic subgroups $H \leqslant G_1$ and
$K \leqslant G_2$. Fix an isomorphism $\varphi \colon H \rightarrow K$. The free product of $G_1$ and $G_2$ with amalgamation of $H$ and $K$ by the isomorphism $\varphi$ is the quotient of $G_1*G_2$ by the normal closure of the set $\{\varphi(h)h^{-1} \mid h \in H \}$. We will refer to this factor group briefly as the \textit{amalgamated product} and have the following notations:
\[
\langle G_1 * G_2 \mid h=\varphi(h), h \in H\rangle,\phantom{aa}
G_1 *_{H=K} G_2,\phantom{aa}
G_1 *_H G_2.
\]
\end{defn}

\section{On the structure of B-B groups}\label{bb structure}

As stated earlier the aim of this article is to describe the structure of Bestvina-Brady groups and we do this, up to some extent, in this section. 
However, we start by focusing at a class of B-B groups that are in fact isomorphic to some RAAG. 
These type of B-B groups can be recognized from the graph defining the ambient RAAG. Then we move to the \textit{iterated amalgamated product} structure having the very first factor group isomorphic to an arbitrary RAAG and other factors isomorphic to $\Z^2$ (which are also RAAGs). Most importantly, the iterated amalgamated product structure of the B-B group can also be derived from the graph defining its ambient RAAG. 

\subsection{Isomorphism between RAAGs and B-B groups}

Since the only relations in B-B groups are the commuting relations, for a given $\Gamma$ it is natural to ask whether $H_\Gamma$ isomorphic to $A_{\Gamma'}$ for some finite simplicial graph $\Gamma'$.
This question was first considered by Papadima and Suciu in~\cite{PS}. 
They also constructed a family of B-B groups not isomorphic to any RAAGs.

We briefly review their result; first we recall the definition of \textit{special} and \textit{extra-special triangulation} (\Cref{triangulation}).

\begin{defn}
A triangulation  of the $2$-disk $D^2$ is said to be \textit{special} if it can be obtained from a triangle by adding one triangle at a time, along a unique boundary edge.
A triangulation of $D^2$ is called \textit{extra-special} if it is obtained from a special triangulation, by adding one triangle along each boundary edge (see Figure~\ref{triangulation}).
\end{defn}

\begin{figure}[H]
    \centering
   \begin{tikzpicture}
    \draw (0,0) -- (3,0) -- (3,3) -- (0,3) -- (0,0);
   \draw (3,0) -- (0,3);
    \draw [->,snake=snake] (4,2) -- (5,2);
    \draw (8,0) -- (10,2) -- (8,4) -- (6,2) -- (8,0);
    \draw (7,1) -- (9,1) -- (9,3) -- (7,3) -- (7,1);
    \draw (7,3) -- (9,1);
\node [fill=black!100,circle,scale=0.3,draw] at (0,0) {};
\node [fill=black!100,circle,scale=0.3,draw] at (3,0) {};
\node [fill=black!100,circle,scale=0.3,draw] at (3,3) {};
\node [fill=black!100,circle,scale=0.3,draw] at (0,3) {};
\node [fill=black!100,circle,scale=0.3,draw] at (8,0) {};
\node [fill=black!100,circle,scale=0.3,draw] at (10,2) {};
 \node [fill=black!100,circle,scale=0.3,draw] at (8,4) {};
 \node [fill=black!100,circle,scale=0.3,draw] at (6,2) {};
\node [fill=black!100,circle,scale=0.3,draw] at (7,1) {};
\node [fill=black!100,circle,scale=0.3,draw] at (9,1) {};
\node [fill=black!100,circle,scale=0.3,draw] at (9,3) {};
\node [fill=black!100,circle,scale=0.3,draw] at (7,3) {};
    \end{tikzpicture}
    \caption{Building an extra-special triangulation of the disk.}
    \label{triangulation}
\end{figure}
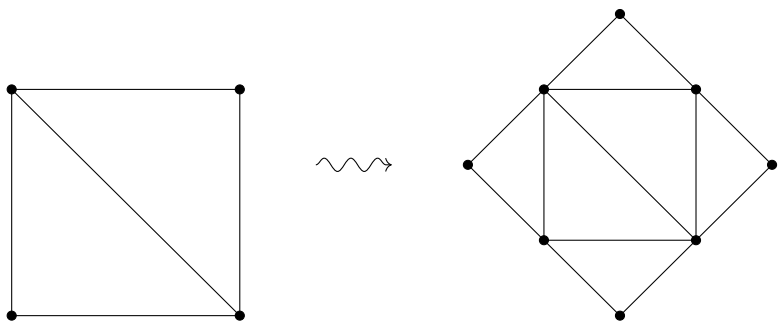

\begin{prop}[{\cite[Proposition 9.4]{PS}}]\label{prop ps}
Let $\Gamma$ be the $1$-skeleton of an extra-special triangulation of
$D^2$. 
Then the corresponding Bestvina--Brady group $H_\Gamma$ is not isomorphic to any Artin group.
\end{prop}

Papadima and Suciu \cite{PS} also describe an explicit presentation of $H_\Gamma$, where $\Gamma$ is the $1$-skeleton of a special triangulation of $D^2$.

\begin{lem}
Let $\Gamma = (V, E)$ be the $1$-skeleton of a special triangulation of $D^2$, then:
\begin{itemize}
    \item[(i)] $2|V| - |E| = 3$,
    \item[(ii)] $H_\Gamma$ admits a presentation with $|V|-1$ generators and $|V| - 2$ commutator relators.
\end{itemize}
\end{lem}



Now we identify a class of graphs such that the corresponding B-B group is isomorphic to some RAAG.
The complete classification of graphs such that the corresponding (finitely presented) B-B group is a RAAG can be found in the recent paper of Chang and Ruffoni.  
In particular, they prove that such graphs admit a \emph{tree $2$-spanner} \cite[Theorem A]{Chang-Ruffoni}. 
Since we do not need the full extent of their result we provide a proof, for the benefit of the reader, of the sufficiency condition without introducing any more technical definitions. 
The authors sincerely thank Chang and Ruffoni for bringing their work to our notice.

\begin{thm}\label{iso thm}
Let $\Gamma$ be a finite simplicial graph such that $\triangle_\Gamma$ is simply-connected and
$\Gamma$ has a spanning tree $T$ such that each triangle of $\Gamma$ has $2$ edges in $E(T)$. 
Then $H_\Gamma \cong A_{\Gamma'}$ for some finite simplicial graph $\Gamma'$.
\end{thm}

\begin{proof}
Let $T$ be a spanning tree of $\Gamma$
such that each triangle of $\Gamma$ has exactly $2$ edges in $E(T)$. 
Let $E(T) = \{e_1, e_2, \ldots , e_n \}$. 
From \Cref{Suciu-Papadima presentation}, $E(T)$ corresponds to a generating set of $H_\Gamma$.
Since $H_\Gamma$ is finitely presented and that each triangle has $2$ edges in $E(T)$, each relator is of the form $[e_i, e_j]$. 

We construct the graph $\Gamma'$ as follows: the vertex set $V(\Gamma')$ is $E(T)$; and two vertices in $V(\Gamma')$ are adjacent whenever the corresponding edges form a triangle in $\Gamma$, i.e., they commute in $H_\Gamma$. Let us denote $V(\Gamma')=\{v'_1, v'_2, \ldots, v'_n \}$, and let $\varphi \colon E(T) \rightarrow V(\Gamma')$ be a map sending each $e_i$ to $v_i$ for $i=1, 2, \ldots, n$. Clearly, $\varphi$ is a bijection. By the construction of $\Gamma'$, we have
$\varphi([e_i, e_j]) = [v'_i, v'_j]$ and $\varphi^{-1}([v'_i, v'_j]) = [e_i , e_j]$. Thus, $\varphi$ is our desired isomorphism between $H_\Gamma$ and $A_{\Gamma'}$.
\end{proof}


\begin{exam}
Let $\Gamma = v\vee\Gamma'$ be a cone graph with $v$ as its dominating vertex. 
A spanning tree can be chosen to be the star graph consisting of consists in $v$, $V(\Gamma')$ and all the edges joining $v$ to each vertex of $\Gamma'$. 
Every edge in $\Gamma'$ and the two edges that connect each of its boundary vertices to $v$ form a triangle. 
Hence, $H_{\Gamma} \cong A_{\Gamma'}$. 
This example has also appeared in \cite[Example 2.5]{PS}.
\end{exam}

\begin{exam}
Let $\Gamma$ be the graph as shown in \Cref{favourable graph}. 
Let $T$ be the spanning tree with edges $e_1$, $e_2$, $e_3$, $e_4$ and $e_5$. 
By \Cref{iso thm}, $H_\Gamma$ is isomorphic to $A_{\Gamma'}$, where $\Gamma'$ is a line graph with vertices $w_1, w_2, w_3, w_4 \text{ and } w_5$ as shown
on the right hand side of \Cref{favourable graph}. 
Note that, $\Gamma$ is not a cone graph, however, $H_\Gamma$ is still isomorphic to $A_{\Gamma'}$.

\begin{figure}[H]
   $$
  \begin{array}{cc}
    \begin{tikzpicture}[shorten >=1pt,node distance=20cm,auto]
    \tikzset{
    edge/.style={draw=black,postaction={on each segment={mid arrow=black}}}
} 
\node[fill=black!100, state, scale=0.10, vrtx/.style args = {#1/#2}{label=#1:#2}] (1) [vrtx=below/$v_1$] {};
\node[fill=black!100, state, scale=0.10, vrtx/.style args = {#1/#2}{label=#1:#2}] (2) [vrtx=below/$v_2$] [right of = 1] {};
\node[fill=black!100, state, scale=0.10, vrtx/.style args = {#1/#2}{label=#1:#2}] (3) [vrtx=below/$v_3$] [right of = 2] {};
\node[fill=black!100, state, scale=0.10, vrtx/.style args = {#1/#2}{label=#1:#2}] (4) [vrtx=above/$v_4$] [above right of = 1] {};
\node[fill=black!100, state, scale=0.10, vrtx/.style args = {#1/#2}{label=#1:#2}] (5) [vrtx=above/$v_5$] [right of = 4] {};
\node[fill=black!100, state, scale=0.10, vrtx/.style args = {#1/#2}{label=#1:#2}] (6) [vrtx=above/$v_6$] [right of = 5] {};

\draw[edge] (1) -- (4) node[midway, above left] {$e_1$};
\draw[edge] (4) -- (2) node[midway, above right] {$e_2$};
\draw[edge] (2) -- (5) node[midway, below right] {$e_3$};
\draw[edge] (5) -- (3) node[midway, above right] {$e_4$};
\draw[edge] (3) -- (6) node[midway, below right] {$e_5$};
\draw[edge] (4) -- (5) -- (6);
\draw[edge] (1) -- (2) -- (3);
 \end{tikzpicture}
       & 
\phantom{aaa}       
\begin{tikzpicture}[shorten >=1pt,node distance=15cm,auto]
  
\node[fill=black!100, state, scale=0.10, vrtx/.style args = {#1/#2}{label=#1:#2}] (1) [vrtx=below/$w_1$] {};
\node[fill=black!100, state, scale=0.10, vrtx/.style args = {#1/#2}{label=#1:#2}] (2) [vrtx=below/$w_2$] [right of = 1] {};
\node[fill=black!100, state, scale=0.10, vrtx/.style args = {#1/#2}{label=#1:#2}] (3) [vrtx=below/$w_3$] [right of = 2] {};
\node[fill=black!100, state, scale=0.10, vrtx/.style args = {#1/#2}{label=#1:#2}] (4) [vrtx=below/$w_4$] [right of = 3] {};
\node[fill=black!100, state, scale=0.10, vrtx/.style args = {#1/#2}{label=#1:#2}] (5) [vrtx=below/$w_5$] [right of = 4] {};
\draw (1) -- (2) -- (3) -- (4) -- (5);
\end{tikzpicture}
  \end{array}
  $$
 \caption{A graph whose corresponding B-B group is a RAAG.}
    \label{favourable graph}
\end{figure}
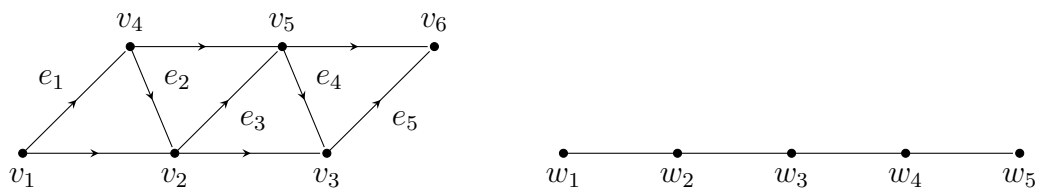
\end{exam}

\begin{nt}\label{imp note on non-isomorphic}
Recall that two RAAGs $A_{\Gamma_1}$ and $A_{\Gamma_2}$ are isomorphic if and only if their underlying finite simplicial graphs $\Gamma_1$ and $\Gamma_2$ are isomorphic. 
However, this is not the case for B-B groups. 
For example, consider any two non-isomorphic trees on $n$ vertices for $n \geq 3$. 
In both the cases the corresponding B-B group is $F_{n-1}$. 
Let $\Gamma_1$, $\Gamma_2$ be cone graphs on $\Gamma'_1 \text{ and } \Gamma'_2$ respectively, then $\Gamma_1 \cong \Gamma_2$ if and only if $\Gamma'_1 \cong \Gamma'_2$. 
So if we restrict ourselves to cone graphs, then $\Gamma_1 \cong \Gamma_2$ if and only if $H_{\Gamma_1} \cong H_{\Gamma_2}$.
See \cite[Corollary 1]{Chang-Ruffoni} for the general result. 
\end{nt}

\subsection{B-B groups as an iterated amalgamated product}
In this subsection we will mainly focus on the triangles of $\Gamma$. 
More precisely, we will concentrate on how a particular triangle $\triangle$ of $\Gamma$ intersects its edge-set complement (see Def.~\ref{complement}). 
We introduce the notions of \textit{favourable} and \textit{unfavourable} triangles with respect to the chosen spanning tree $T$. Accordingly we also introduce the notions of \textit{favourable} and \textit{unfavourable} graphs.

\begin{defn}\label{complement}
Let $\Gamma$ be a connected graph and $\triangle$ be a triangle of $\Gamma$. Let $\Gamma'$ be the graph with $V(\Gamma')=V(\Gamma)$ and $E(\Gamma')=E(\Gamma)\setminus E(\triangle)$. Let $S$ be the isolated vertices of $\Gamma'$ and $V^c=V(\Gamma) \setminus S$. The \textit{edge-set ~complement} of $\triangle$ is the induced subgraph of $\Gamma$ generated by the set of vertices $V^c$. We will denote the edge-set complement of $\triangle$ in $\Gamma$ by $\triangle_\Gamma^c$.

\end{defn}

\begin{defn}
A triangle $\triangle$ of $\Gamma$ is said to be an \textit{internal triangle} if its intersection with the edge-set complement $\triangle_\Gamma^c$ is neither one vertex nor one edge.
A triangle $\triangle$ of $\Gamma$ is said to be a \textit{strictly internal triangle} if it is contained in an induced $K_n$, $n \geq 4$.
\end{defn}

In Figure~\ref{fig:graph with int. tri.} the triangles $\triangle(v_2, v_4, v_5)$ and $\triangle(v_2, v_3, v_4)$ are the internal triangles.
\begin{figure}[H]
    $$
    \begin{array}{ccc}
\begin{tikzpicture}[shorten >=1pt,node distance=17.5cm,auto]
\tikzset{
    edge/.style={draw=black,postaction={on each segment={mid arrow=black}}}
}
\node[fill=black!100, state, scale=0.10, vrtx/.style args = {#1/#2}{label=#1:#2}] (1) [vrtx=left/$v_6$] {};

\node[fill=black!100, state, scale=0.10, vrtx/.style args = {#1/#2}{label=#1:#2}] (2) [vrtx=above/$v_4$] [ below right of = 1] {};

\node[fill=black!100, state, scale=0.10, vrtx/.style args = {#1/#2}{label=#1:#2}] (7) [vrtx=right/$v_7$] [right of = 2] {};

\node[fill=black!100, state, scale=0.10, vrtx/.style args = {#1/#2}{label=#1:#2}] (3) [vrtx=left/$v_5$] [ below left of = 1] {};

\node[fill=black!100, state, scale=0.10, vrtx/.style args = {#1/#2}{label=#1:#2}] (4) [vrtx=right/$v_3$] [ below right of = 2] {};

\node[fill=black!100, state, scale=0.10, vrtx/.style args = {#1/#2}{label=#1:#2}] (5) [vrtx=below/$v_2$] [ below left of = 2] {};

\node[fill=black!100, state, scale=0.10, vrtx/.style args = {#1/#2}{label=#1:#2}] (6) [vrtx=left/$v_1$] [ below left of = 3] {};
\draw[edge] (5) -- (6) node[midway, below] {$e_1$};
\draw[edge] (5) -- (3);
\draw[edge] (5) -- (2);
\draw[edge] (5) -- (4);
\draw[edge] (1) -- (2) node[midway, right] {$e_4$};
\draw[edge] (6) -- (3) node[midway, left] {$e_2$};
\draw[edge] (3) -- (2);
\draw[edge] (3) -- (1) node[midway, left] {$e_3$};
\draw[edge] (2) -- (4);
\draw[edge] (2) -- (7) node[midway, above] {$e_5$};
\draw[edge] (7) -- (4) node[midway, right] {$e_6$};
\end{tikzpicture}
  

  \end{array}
  $$
\caption{A graph with internal triangles.}
\label{fig:graph with int. tri.}
\end{figure}
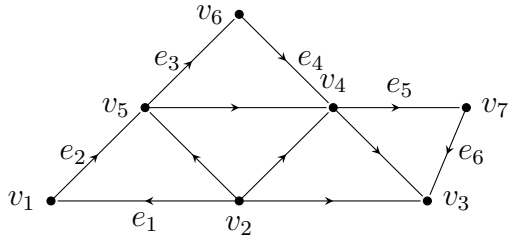

It is straightforward to note that all strictly internal triangles are internal.


\begin{defn}
Let $T$ be a spanning tree of $\Gamma$. 
A triangle $\triangle$ of $\Gamma$ is a \textit{favourable triangle}  with respect to $T$ if, either it has exactly $2$ edges in $E(T)$ or it is strictly internal.  
Otherwise, we say that $\triangle$ is \textit{unfavourable}.
\end{defn}

\begin{nt}
A graph $\Gamma$ can have several spanning trees. 
So, it may very well happen that a triangle $\triangle$ is \textit{favourable} in one spanning tree $T_1$ of $\Gamma$ but not favourable in another spanning tree $T_2$ of $\Gamma$. 
Hence to talk about favourable triangles we have to fix a spanning tree $T$ of $\Gamma$. 
We will choose a spanning tree  with the \emph{maximal} number of \textit{favourable} triangles or equivalently, the number of \textit{unfavourable} triangles is \textit{minimal}.

\end{nt}

Since by definition every strictly internal triangle of $\Gamma$ is favourable, strictly internal triangles of $\Gamma$ are always favourable with respect to any choice of  spanning tree of $\Gamma$. Hence, we define \textit{favourable graphs} as follows.

\begin{defn}
Let $\Gamma$ be a finite simplicial graph such that the flag complex on $\Gamma$ is simply-connected. 
$\Gamma$ is said to be a \textit{favourable graph} if there is a spanning tree $T$ such that each triangle of $\Gamma$ is favourable with respect to $T$.
On the other hand, if for every spanning tree $T$ of $\Gamma$, there exists at least one unfavourable triangle 
then $\Gamma$ is called an \textit{unfavourable graph}.
\end{defn}

The graph in \Cref{favourable graph} is favourable and the graph in \Cref{fig:exam} is unfavourable. 
In light of \Cref{iso thm}, if $\Gamma$ is a favourable graph, then the corresponding B-B group is isomorphic to a RAAG. 
However, it is important to note that an unfavourable graph doesn't mean that the corresponding B-B group is not isomorphic to any RAAG. 

\begin{obs}\label{crucial obs}
Any non-internal triangle can have $1$ edge or $2$ edges in any spanning tree. On the other hand, we know that if the triangle has $2$ edges in the spanning tree, then it is favourable with respect to that particular spanning tree.

Let $T$ be a chosen spanning tree of $\Gamma$ and $\triangle$ be a non-internal triangle. 
If the triangle $\triangle$ is unfavourable with respect to $T$, then it has exactly one edge in $T$. 
\end{obs}

Now some preparatory results.

\begin{lem}\label{one vertex lem}
If a triangle $\triangle$ of $\Gamma$ intersects the corresponding edge-set complement $\triangle_\Gamma^c$ in one vertex, then it is a favourable triangle irrespective of the choice of a spanning tree.
\end{lem}



\begin{proof}
Let the triangle $\triangle$ intersect $\triangle_\Gamma^c$ in one vertex, say $v$. 
Then $v$ is a cut vertex. 
Hence, for any spanning tree $T$, $\triangle$ has exactly two edges in $E(T)$. 
\end{proof}

The following result is a direct consequence of  Dicks-Leary\cite{DL} presentation. We provide a proof for the benefit of the reader. 

\begin{lem}\label{edge element}
Let $\Gamma$ be a finite simplicial graph with  simply-connected flag complex. 
Then there is a group element in $H_{\Gamma}$ for each  $e\in E(\Gamma)$.  
\end{lem}

\begin{proof}
By abusing the notation we let $e$ denote the edge as well as the corresponding group element. 
Let us first fix a spanning tree $T$ of $\Gamma$, and $\iota(e)=u \text{ and } \tau(e)=v$. 
If $e\in E(T)$, then there is nothing to prove. 
Let us suppose that $e$ is an edge not in $T$. 
Also, let $A_\Gamma$ be the RAAG defined on $\Gamma$. 
Picking a path $e_1^{\epsilon_1}, \ldots, e_r^{\epsilon_r}$ in $T$ connecting $u$ to $v$, we see that $\iota(e)=\iota({e_1}^{\epsilon_1} \ldots  {e_r}^{\epsilon_r}) \in A_\Gamma$, $\epsilon_i=\pm 1$. 
Clearly, $e={e_1}^{\epsilon_1} \ldots  {e_r}^{\epsilon_r} \in H_\Gamma$.
\end{proof}

\begin{prop}\label{basic decomp}
Let $T$ be a spanning tree of $\Gamma$ and let $\triangle$ be an unfavourable triangle of $\Gamma$ with respect to $T$ such that it intersects the edge-set complement $\triangle_\Gamma^c$ exactly in one edge. Then $H_{\Gamma} \cong H_{\triangle} *_{\ZZ}  H_{\triangle_\Gamma^c}$.
\end{prop}

\begin{proof}
The Bestvina--Brady group $H_\Gamma$ is finitely presented and its generators are the edges of the spanning tree $T$.
Let $E(\triangle)=\{ e_1, e_2, e_3\}$. Since $\triangle$ intersects $\triangle_\Gamma^c$ exactly in one edge, $\triangle$ is non-internal by definition. On the other hand, $\triangle$ is unfavourable. Hence, by Observation~\ref{crucial obs}  we may assume that $e_1 \in E(T)$ and $e_2, e_3 \notin E(T)$. 
We know that $H_\triangle$, the Bestvina--Brady group associated to $\triangle$ is a copy of $\ZZ^2$ with the following presentation:

\begin{equation*}
    H_\triangle = \bigl \langle e_1, e_2 \mid e_1e_2=e_2e_1 \bigr\rangle .
\end{equation*}

Let $e_2$ be the edge common to both $\triangle$ and $\triangle_\Gamma^c$. 
To avoid any confusion, we will denote $e_2$ by $\overline{e_2}$ when seen as an edge of $\triangle_\Gamma^c$.
Note that the restriction of $T$ to $\triangle_\Gamma^c$ denoted by $T'$ is a spanning tree of $\triangle_\Gamma^c$.
As, $\overline{e_2} \notin E(T')$, it is not a generator of $H_{\triangle_\Gamma^c}$.
However by Lemma~\ref{edge element}, $\langle \overline{e_2}\rangle$ is a subgroup of $H_{\triangle_\Gamma^c}$. 
On the other hand $e_2$ is a generator of $H_\triangle$ and of course, $\langle e_2 \rangle$ is a subgroup of $H_\triangle$. 

Note that all generators and relators of $H_\Gamma$ are present in $H_{\triangle_\Gamma^c}$, except the generator $e_1$ and the relations involving $e_1$. 
So our first step is to include $e_1$ in some extension of $H_{\triangle_\Gamma^c}$; which can be done by taking the free product of $H_{\triangle_\Gamma^c}$ and $H_\triangle$. 
Thus $H_\triangle * H_{\triangle_\Gamma^c}$ is generated by all the generators of $H_\Gamma$ and $e_2$. 
So our next step is to get rid of $e_2$. 
Let us consider the following amalgamated product:

\[
H_\triangle *_{\langle e_2\rangle =\langle \overline{e_2} \rangle} H_{\triangle_\Gamma^c}
= H_\triangle *_{\langle e_2\rangle} H_{\triangle_\Gamma^c}.
\]

Thus we include $e_1$ in the generating set and at the same time we discard the generator $e_2$.
In the amalgamated free product one of the relations is $e_2=\overline{e_2}$ and $\overline{e_2}$ can be expressed as a word in terms of the other generators of $H_{\triangle_\Gamma^c}$; leaving $e_2$ redundant. 
Now comparing the presentations of $H_\Gamma$ and $H_\triangle *_{\langle e_2\rangle} H_{\triangle_\Gamma^c}$, it is clear that, $H_\Gamma \cong H_\triangle *_{\langle e_2\rangle} H_{\triangle_\Gamma^c} \cong H_{\triangle} *_{\ZZ}  H_{\triangle_\Gamma^c}$.
This completes the proof. 
\end{proof}

We note that, if $\Gamma = \Gamma_1 \cup \Gamma_2$, where $\Gamma_1, \Gamma_2$ are finite simplicial graphs and their flag complexes are simply-connected; moreover, if  $\Gamma_3 = \Gamma_1 \cap \Gamma_2$ is a connected induced subgraph of $\Gamma$, then $H_\Gamma = H_{\Gamma_1} *_{H_{\Gamma_3}} H_{\Gamma_2}$. This is proved in \cite[Proposition 3.5]{Chang Thesis} and discussed in \cite{Chang Dehn functions}.

For the \textit{iterated amalgamated product} structure, we are mainly interested in  \textit{unfavourable} graphs. 
We want to have an iterated amalgamation structure such that the factor groups are RAAGs. 
Such a decomposition is helpful in understanding various properties that are known for RAAGs and can pass through the iteration of amalgamations. 
The class of graphs for which such an iterated amalagamation exists is the following:

\begin{defn}
Let $\mathcal{G}$ be a family of finite, simplicial, unfavourable graphs $\Gamma$ with simply connected flag complex which, in addition, have a spanning tree $T$ such that all the internal triangles are favourable with respect to $T$.
\end{defn} 

Now we describe the structure of B-B groups associated to graphs in $\mathcal{G}$. 

\begin{thm}\label{structure thm}

Let $\Gamma \in \mathcal{G}$. Then $H_\Gamma$ splits as an iterated amalgamated product of a right-angled Artin group and finitely many copies of $\ZZ^2$ and in each iteration the amalgamation is over an infinite cyclic group.
\end{thm}

\begin{proof}
Since $\Gamma$  is in $\mathcal{G}$, $\Gamma$ has spanning trees with respect to which none of the internal triangles is unfavourable. 
Among those we can choose a spanning tree $T$ with respect to which $\Gamma$ has minimal number of unfavourable triangles.
By \Cref{crucial obs}, each of the unfavourable triangles has exactly one edge in the spanning tree.
Let $\{ \triangle_1, \triangle_2, \ldots, \triangle_n \}$ be the set of unfavorable triangles and denote by $e_i$ that particular edge of $\triangle_i$ which is in $E(T)$. 
We denote the other two edges of $\triangle_i$ by $f_i \text{ and } f'_i$; note that, these edges are not in $E(T)$. 
It follows from \Cref{one vertex lem} 
that $\triangle_i$ intersects the edge-set complement ${(\triangle_i)}_\Gamma^c$ in exactly one edge, say $f_i$ for each $i=1, 2, \dots, n$.

We already know that the B-B group $H_{\triangle_i}$ associated to $\triangle_i$ is finitely presented with the following presentation:

\begin{equation}\label{triangle equation}
 H_{\triangle_i} = \bigl\langle e_i, f_i | e_if_i=f'_i=f_ie_i \bigr\rangle \cong \ZZ^2; \phantom{aaa} i=1, 2, \ldots, n.
\end{equation}

Consider the subgraphs defined as follows: 
\begin{equation}\label{structure of gamma}
  \Gamma_1:={(\triangle_1)}_\Gamma^c \text{ and } \Gamma_i:= {(\triangle_i)}_{\Gamma_{i-1}}^c, \phantom{aa} i=2, 3, \ldots, n.  
\end{equation}
Also note that, $\Gamma_1$ is an induced subgraph of $\Gamma$ and $\Gamma_i$ is an induced subgraph of $\Gamma_{i-1}$.

By \cref{basic decomp} we have,
\begin{equation}\label{H-triangle}
 H_{\Gamma} = H_{\Gamma_1} *_{\langle f_1 \rangle} 
 H_{\triangle_1}
\end{equation}
and
\begin{equation}\label{H-gamma}
 H_{\Gamma_{i-1}} = H_{\Gamma_i} *_{\langle f_i \rangle} 
 H_{\triangle_i} \phantom{aa}\hbox{for~} i=2, 3, \ldots, n.
\end{equation}

Note that we choose the restriction of $T$ to $\Gamma_i$ as the spanning tree for $\Gamma_i$ and denote it by $T_i$, for $i=1, 2, 3, \ldots, n$. 
From the construction and \Cref{structure of gamma} it is clear that the graph $\Gamma_n:={(\triangle_n)}_{\Gamma_{n-1}}^c$ is a \textit{favourable graph}. 
So, $H_{\Gamma_n}$ is isomorphic to $A_{\overline{\Gamma}}$ for some finite simplicial graph $\overline{\Gamma}$, see \Cref{iso thm}. 
Recall that, the edges of $T_n$ are the vertices of $\overline{\Gamma}$ and two such vertices
are adjacent whenever the corresponding edges form a triangle in $\Gamma_n$.



We can now express the B-B group $H_{\Gamma}$ as the following (iterated) amalgamated product. 

\begin{equation*}
  H_{\Gamma}=\Bigg( \Big( (A_{\overline{\Gamma}} *_{\langle f_n \rangle}H_{\triangle_n})*_{\langle f_{n-1}\rangle} H_{\triangle_{n-1}} \Big)*_{\langle f_{n-2}\rangle} H_{\triangle_{n-2}} \bigg) \cdots*_{\langle f_{n-i} \rangle}H_{\triangle_{n-i}} \cdots \Bigg)*_{\langle f_1 \rangle} H_{\triangle_1}.
\end{equation*}
This completes the proof.
\end{proof}

To characterize the class of graphs which belong to the family $\mathcal{G}$ we use the notion of \textit{separating clique} from the literature (see \cite{GH} for more details). 




\begin{defn}
Let $\Gamma_0$ be an induced subgraph of $\Gamma$. $\Gamma_0$ is said to be \textit{separating} if the induced subgraph spanned by the vertices $V(\Gamma) \setminus V(\Gamma_0)$ has more connected components than $\Gamma$. If $\Gamma_0$ is a complete graph on $n$ vertices, then $\Gamma_0$ is called a \textit{separating clique} of $\Gamma$ and is denoted by $K_n$.
\end{defn}

Here is a partial characterization of graphs in $\mathcal{G}$.

\begin{thm}\label{the important graphs 1}
   Let $\Gamma$ be a finite simplicial connected unfavourable graph such that $\triangle_\Gamma$, the associated flag complex, is simply connected. If $\Gamma$ has a separating clique $K_n$, $n \geq 3$, then $\Gamma \in \mathcal{G}$.
\end{thm}

\begin{proof}
 Let us suppose that $\Gamma$ has $\ell \geq 1$ many internal triangles, say $\triangle_1, \triangle_2, \ldots, \triangle_\ell$. 
Our claim is that each $\triangle_i$ is a favourable triangle, for $i=1, 2, \ldots, \ell$.
We will prove our claim by induction on $\ell$. 

First consider the case $\ell=1$, i.e., only one internal triangle, say $\triangle$.  
We can choose a tree $T'$ by taking any two edges of $\triangle$ and then we expand $T'$ to a spanning tree of $\Gamma$. 
This makes $\triangle$ a favourable triangle with respect to $T$.


Assume that the claim holds for all graphs which have fewer than $\ell$ internal triangles. 
As $\Gamma$ has a separating clique $K_n$, then $\Gamma \setminus K_n = \Gamma_1 \sqcup \Gamma_2$ where $\Gamma_1$ and $\Gamma_2$ are each non-empty and share no vertices.
Let $V(K_n)= \{u_1, u_2, \ldots, u_n \}$. 
There exists an $u_i$ which is adjacent to a vertex of $\Gamma_1$ and the same thing holds for $\Gamma_2$. 
Without loss of generality we can assume that $u_1$ is adjacent to $\Gamma_1$ and $u_2$ is adjacent to a vertex of $\Gamma_2$. 
Now let $\Gamma_1'$ be the induced subgraph spanned by $V(\Gamma_1) \cup \{u_1\}$ and $\Gamma_2'$ be the induced subgraph spanned by $V(\Gamma_2) \cup \{u_2\}$. 
Then both $\Gamma_1'$ and $\Gamma_2'$ have less than $\ell$ internal triangles. 
By the induction hypothesis, $\Gamma_1'$ and $\Gamma_2'$  have spanning trees $T_1'$ and $T_2'$ respectively for which all the internal triangles are favourable. 
Since $K_n$ is a complete graph, we can choose the spanning tree of $K_n$ as a star graph. 
{Let $T_{K_n}$ be the spanning tree of $K_n$. 
Hence for $T_{K_n}$ every triangle of $K_n$ is favourable}. 
Let us consider $T=T_1' \cup T_{K_n} \cup T_2'$. Then $T$ is the desired spanning tree of $\Gamma$ for which all the internal triangles are favourable.   
\end{proof}

\begin{obs}\label{the important graphs 2}
If $\Gamma$ is a $1$-skeleton of an extra-special triangulation of the $2$-disk such that the underlying special triangulation is favourable, then $\Gamma \in \mathcal{G}$. 
The fact holds because of the following reason: the internal triangles of an extra-special triangulation of $D^2$ are the triangles of the special triangulation on which the extra-special triangulation is built by attaching one triangle along each boundary edge. 
According to our hypothesis, this special triangulation is favourable, hence it possesses a spanning tree $T'$ containing $2$ edges from each triangle of the special triangulation. 
Now if we construct the spanning tree $T$ of $\Gamma$ by expanding the tree $T'$, then our observation follows.
\end{obs}

We now look at two examples. 
First we consider the graph in \cref{fig:exam} which is an extra-special triangulation. 
Recall that, from \cite[Proposition 9.4]{PS} it follows that the corresponding B-B group is not isomorphic to a RAAG. 

\begin{exam}
Our choice of spanning tree is $T = \{e_1, \dots, e_5\}$ for the graph in \cref{fig:exam}. Hence the presentation for the B-B groups is
\[H_\Gamma = \langle e_1,  \ldots, e_5 \mid [e_1 , e_2], [e_2 , e_3], [e_3 , e_4], e_5{e_2}^{-1}e_3 = {e_2}^{-1}e_3 e_5 \rangle. \]
Note that the triangle $\triangle_1$ spanned by $\{v_4, v_5, v_6\}$ is unfavourable with respect to $T$ and let $E(\triangle_1) = \{ e_5, f_1, f'_1\}$ where $f_1, f'_1$ denote the respective edges $(v_5, v_4)$ and $(v_5, v_6)$. 
Let $\Gamma_1$ be the edge-set complement of $\triangle_1$. 
It is not hard to see that $\Gamma_1$ is a favourable graph; in fact, $\{e_1, e_2, e_3, e_4\}$ is a spanning tree with respect to which all the three triangles are favourable. 
Consequently, 
\[H_{\Gamma} = H_{\Gamma_1}\ast_{\langle f_1\rangle} H_{\triangle_1} \cong A_{\overline{\Gamma}} \ast_\ZZ \ZZ^2. \]
\end{exam}
Now we consider an example where                
the underlying graph is not an extra-special triangulation and the minimal separating clique is $K_2$. 

\begin{exam}
Let us consider $\Gamma$ in \Cref{example}. 
We choose the following spanning tree $T=\{ e_1, e_2, \ldots, e_{11}\}$. 
Note that there are two non-favourable triangles $\triangle(v_1, v_2, v_5)=~\triangle_1$ and $\triangle(v_{10}, v_{11}, v_9)=\triangle_2$. 
In fact, the reader can verify that for any choice of spanning tree, there will be at least two unfavourable triangles.

We have, $E(\triangle_1)=\{e_1, f_1, f_1'\} \text{ where } e_1 \in~E(T) \text{ and } f_1, f_1' \notin E(T)$ and $E(\triangle_2)=\{e_{10}, f_2, f_2'\} \text{ where } e_{10} \in~E(T) \text{ and } f_2, f_2' \notin~E(T)$. 
Let us denote $\Gamma_1:={\triangle_1}_\Gamma^c$ (see \Cref{example_1}) and $\Gamma_2:={\triangle_2}_{\Gamma_1}^c$ (see \Cref{example_2}). Hence, $\Gamma_1$ is the induced subgraph of $\Gamma$  and $\Gamma_2$ is the induced subgraph of $\Gamma_{1}$. 
We have:
\begin{align*}
H_{\triangle_1} &= \bigl\langle e_1, f_1 | e_1f_1=f'_1=f_1e_1 \bigr\rangle \cong \ZZ^2. \\
H_{\triangle_2} &= \bigl\langle e_{10}, f_2 | e_{10}f_2=f'_2=f_2e_{10} \bigr\rangle \cong \ZZ^2.
\end{align*}

Thus 
\begin{equation}\label{H-triangle again}
 H_{\Gamma} = H_{\Gamma_1} *_{\langle f_1 \rangle} 
 H_{\triangle_1},  
\end{equation}
and,
\begin{equation}\label{H-gamma again}
 H_{\Gamma_{1}} = H_{\Gamma_2} *_{\langle f_2 \rangle} 
 H_{\triangle_2}.
\end{equation}

Note that $\Gamma_2$ is a favourable graph. So, $H_{\Gamma_{2}} \cong A_{\overline{\Gamma}}$ where 
$$A_{\overline{\Gamma}}= \langle w_2, \ldots, w_9, w_{11} \mid [w_2, w_3],[w_2, w_4], [w_4, w_5], [w_7, w_8], [w_8, w_9], [w_9, w_{11}] \rangle.$$ 
Thus finally we have the desired iterated amalgamation structure of $H_\Gamma \colon$

\[H_\Gamma=(A_{\overline{\Gamma}} *_{\langle f_2 \rangle} H_{\triangle_2}) *_{\langle f_1 \rangle} H_{\triangle_1} \cong (A_{\overline{\Gamma}} *_{\ZZ} \ZZ^2)*_{\ZZ} \ZZ^2.\]
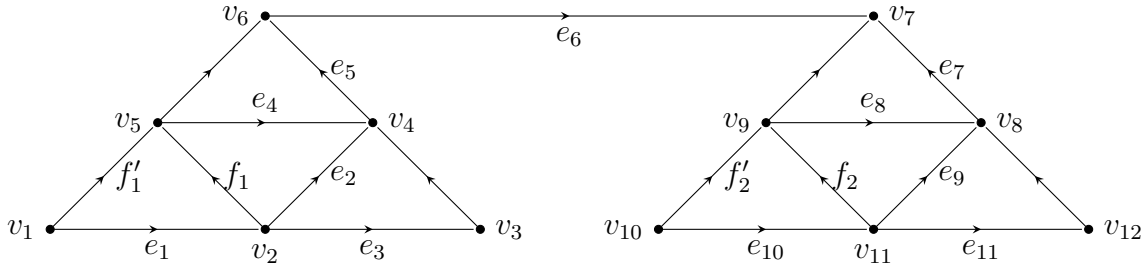
\begin{figure}[H]
  \centering 
\begin{tikzpicture}[shorten >=1pt,node distance=20cm,auto]
\tikzset{
    edge/.style={draw=black,postaction={on each segment={mid arrow=black}}}
}
\node[fill=black!100, state, scale=0.10, vrtx/.style args = {#1/#2}{label=#1:#2}] (1) [vrtx=left/$v_6$] {};

\node[fill=black!100, state, scale=0.10, vrtx/.style args = {#1/#2}{label=#1:#2}] (2) [vrtx=right/$v_4$] [ below right of = 1] {};

\node[fill=black!100, state, scale=0.10, vrtx/.style args = {#1/#2}{label=#1:#2}] (3) [vrtx=left/$v_5$] [ below left of = 1] {};

\node[fill=black!100, state, scale=0.10, vrtx/.style args = {#1/#2}{label=#1:#2}] (4) [vrtx=right/$v_3$] [ below right of = 2] {};

\node[fill=black!100, state, scale=0.10, vrtx/.style args = {#1/#2}{label=#1:#2}] (5) [vrtx=below/$v_2$] [ below left of = 2] {};

\node[fill=black!100, state, scale=0.10, vrtx/.style args = {#1/#2}{label=#1:#2}] (6) [vrtx=left/$v_1$] [ below left of = 3] {};

\phantom{\node[fill=black!100, state, scale=0.10, vrtx/.style args = {#1/#2}{label=#1:#2}] (13) [vrtx=right/$v_{13}$] [right of = 1] {};}

\phantom{\node[fill=black!100, state, scale=0.10, vrtx/.style args = {#1/#2}{label=#1:#2}] (14) [vrtx=right/$v_{14}$] [right of = 13] {};}

\phantom{\node[fill=black!100, state, scale=0.10, vrtx/.style args = {#1/#2}{label=#1:#2}] (15) [vrtx=right/$v_{15}$] [right of = 14] {};}


\draw[edge] (6) -- (5) node[midway, below] {$e_1$};
\draw[edge] (5) -- (3) node[midway, right] {$f_1$};
\draw[edge] (5) -- (2) node[midway, right] {$e_2$};
\draw[edge] (5) -- (4) node[midway, below] {$e_3$};
\draw[edge] (2) -- (1) node[midway, right] {$e_5$};
\draw[edge] (6) -- (3) node[midway, right] {$f_1'$};
\draw[edge] (3) -- (2) node[midway, above] {$e_4$};
\draw[edge] (3) -- (1);
\draw[edge] (4) -- (2);

\node[fill=black!100, state, scale=0.10, vrtx/.style args = {#1/#2}{label=#1:#2}] (7) [vrtx=right/$v_7$] [right of = 15] {};

\node[fill=black!100, state, scale=0.10, vrtx/.style args = {#1/#2}{label=#1:#2}] (8) [vrtx=right/$v_8$] [ below right of = 7] {};

\node[fill=black!100, state, scale=0.10, vrtx/.style args = {#1/#2}{label=#1:#2}] (9) [vrtx=left/$v_9$] [ below left of = 7] {};

\node[fill=black!100, state, scale=0.10, vrtx/.style args = {#1/#2}{label=#1:#2}] (10) [vrtx=right/$v_{12}$] [ below right of = 8] {};

\node[fill=black!100, state, scale=0.10, vrtx/.style args = {#1/#2}{label=#1:#2}] (11) [vrtx=below/$v_{11}$] [ below left of = 8] {};

\node[fill=black!100, state, scale=0.10, vrtx/.style args = {#1/#2}{label=#1:#2}] (12) [vrtx=left/$v_{10}$] [ below left of = 9] {};

\draw[edge] (12) -- (11) node[midway, below] {$e_{10}$};
\draw[edge] (11) -- (9) node[midway, right] {$f_2$};
\draw[edge] (11) -- (8) node[midway, right] {$e_9$};
\draw[edge] (11) -- (10) node[midway, below] {$e_{11}$};
\draw[edge] (8) -- (7) node[midway, right] {$e_{7}$};
\draw[edge] (12) -- (9) node[midway, right] {$f_2'$};
\draw[edge] (9) -- (8) node[midway, above] {$e_8$};
\draw[edge] (9) -- (7);
\draw[edge] (10) -- (8);
\draw[edge] (1) -- (7) node[midway, below] {$e_{6}$};

\end{tikzpicture}
\caption{The graph $\Gamma$.}
\label{example}
\end{figure}

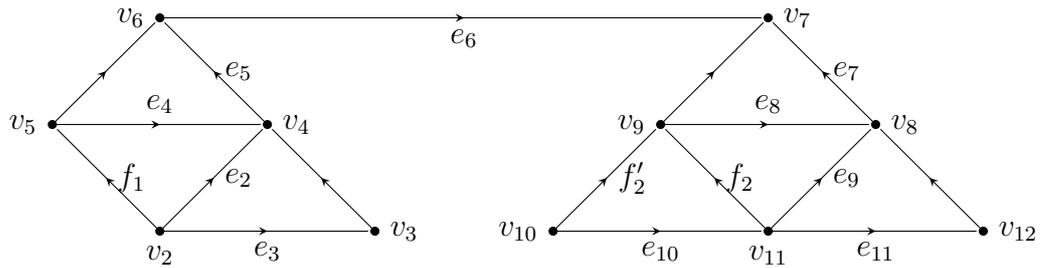
\begin{figure}[H]
  \centering 
\begin{tikzpicture}[shorten >=1pt,node distance=20cm,auto]
\tikzset{
    edge/.style={draw=black,postaction={on each segment={mid arrow=black}}}
}
\node[fill=black!100, state, scale=0.10, vrtx/.style args = {#1/#2}{label=#1:#2}] (1) [vrtx=left/$v_6$] {};

\node[fill=black!100, state, scale=0.10, vrtx/.style args = {#1/#2}{label=#1:#2}] (2) [vrtx=right/$v_4$] [ below right of = 1] {};

\node[fill=black!100, state, scale=0.10, vrtx/.style args = {#1/#2}{label=#1:#2}] (3) [vrtx=left/$v_5$] [ below left of = 1] {};

\node[fill=black!100, state, scale=0.10, vrtx/.style args = {#1/#2}{label=#1:#2}] (4) [vrtx=right/$v_3$] [ below right of = 2] {};

\node[fill=black!100, state, scale=0.10, vrtx/.style args = {#1/#2}{label=#1:#2}] (5) [vrtx=below/$v_2$] [ below left of = 2] {};

\phantom{\node[fill=black!100, state, scale=0.10, vrtx/.style args = {#1/#2}{label=#1:#2}] (6) [vrtx=left/$v_1$] [ below left of = 3] {};}

\phantom{\node[fill=black!100, state, scale=0.10, vrtx/.style args = {#1/#2}{label=#1:#2}] (13) [vrtx=right/$v_{13}$] [right of = 1] {};}

\phantom{\node[fill=black!100, state, scale=0.10, vrtx/.style args = {#1/#2}{label=#1:#2}] (14) [vrtx=right/$v_{14}$] [right of = 13] {};}

\phantom{\node[fill=black!100, state, scale=0.10, vrtx/.style args = {#1/#2}{label=#1:#2}] (15) [vrtx=right/$v_{15}$] [right of = 14] {};}


\draw[edge] (5) -- (3) node[midway, right] {$f_1$};
\draw[edge] (5) -- (2) node[midway, right] {$e_2$};
\draw[edge] (5) -- (4) node[midway, below] {$e_3$};
\draw[edge] (2) -- (1) node[midway, right] {$e_5$};
\draw[edge] (3) -- (2) node[midway, above] {$e_4$};
\draw[edge] (3) -- (1);
\draw[edge] (4) -- (2);


\node[fill=black!100, state, scale=0.10, vrtx/.style args = {#1/#2}{label=#1:#2}] (7) [vrtx=right/$v_7$] [right of = 15] {};

\node[fill=black!100, state, scale=0.10, vrtx/.style args = {#1/#2}{label=#1:#2}] (8) [vrtx=right/$v_8$] [ below right of = 7] {};

\node[fill=black!100, state, scale=0.10, vrtx/.style args = {#1/#2}{label=#1:#2}] (9) [vrtx=left/$v_9$] [ below left of = 7] {};

\node[fill=black!100, state, scale=0.10, vrtx/.style args = {#1/#2}{label=#1:#2}] (10) [vrtx=right/$v_{12}$] [ below right of = 8] {};

\node[fill=black!100, state, scale=0.10, vrtx/.style args = {#1/#2}{label=#1:#2}] (11) [vrtx=below/$v_{11}$] [ below left of = 8] {};

\node[fill=black!100, state, scale=0.10, vrtx/.style args = {#1/#2}{label=#1:#2}] (12) [vrtx=left/$v_{10}$] [ below left of = 9] {};

\draw[edge] (12) -- (11) node[midway, below] {$e_{10}$};
\draw[edge] (11) -- (9) node[midway, right] {$f_{2}$};
\draw[edge] (11) -- (8) node[midway, right] {$e_9$};
\draw[edge] (11) -- (10) node[midway, below] {$e_{11}$};
\draw[edge] (8) -- (7) node[midway, right] {$e_{7}$};
\draw[edge] (12) -- (9)node[midway, right] {$f_{2}'$};
\draw[edge] (9) -- (8) node[midway, above] {$e_8$};
\draw[edge] (9) -- (7);
\draw[edge] (10) -- (8);
\draw[edge] (1) -- (7) node[midway, below] {$e_{6}$};

\end{tikzpicture}

\caption{The graph $\Gamma_1$.}
\label{example_1}
\end{figure}

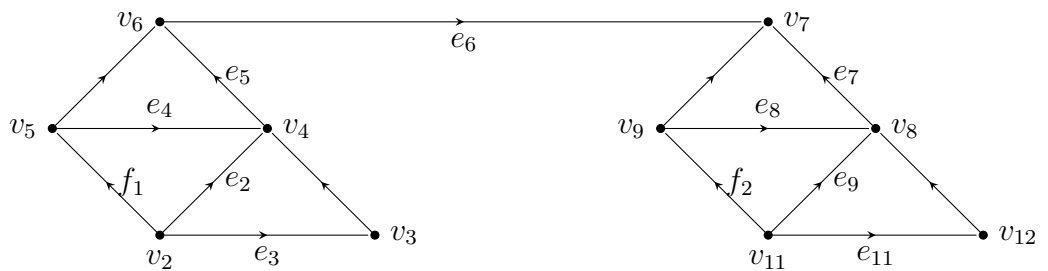
\begin{figure}[H]
  \centering 
\begin{tikzpicture}[shorten >=1pt,node distance=20cm,auto]
\tikzset{
    edge/.style={draw=black,postaction={on each segment={mid arrow=black}}}
}
\node[fill=black!100, state, scale=0.10, vrtx/.style args = {#1/#2}{label=#1:#2}] (1) [vrtx=left/$v_6$] {};

\node[fill=black!100, state, scale=0.10, vrtx/.style args = {#1/#2}{label=#1:#2}] (2) [vrtx=right/$v_4$] [ below right of = 1] {};

\node[fill=black!100, state, scale=0.10, vrtx/.style args = {#1/#2}{label=#1:#2}] (3) [vrtx=left/$v_5$] [ below left of = 1] {};

\node[fill=black!100, state, scale=0.10, vrtx/.style args = {#1/#2}{label=#1:#2}] (4) [vrtx=right/$v_3$] [ below right of = 2] {};

\node[fill=black!100, state, scale=0.10, vrtx/.style args = {#1/#2}{label=#1:#2}] (5) [vrtx=below/$v_2$] [ below left of = 2] {};

\phantom{\node[fill=black!100, state, scale=0.10, vrtx/.style args = {#1/#2}{label=#1:#2}] (6) [vrtx=left/$v_1$] [ below left of = 3] {};}

\phantom{\node[fill=black!100, state, scale=0.10, vrtx/.style args = {#1/#2}{label=#1:#2}] (13) [vrtx=right/$v_{13}$] [right of = 1] {};}

\phantom{\node[fill=black!100, state, scale=0.10, vrtx/.style args = {#1/#2}{label=#1:#2}] (14) [vrtx=right/$v_{14}$] [right of = 13] {};}

\phantom{\node[fill=black!100, state, scale=0.10, vrtx/.style args = {#1/#2}{label=#1:#2}] (15) [vrtx=right/$v_{15}$] [right of = 14] {};}


\draw[edge] (5) -- (3) node[midway, right] {$f_1$};
\draw[edge] (5) -- (2) node[midway, right] {$e_2$};
\draw[edge] (5) -- (4) node[midway, below] {$e_3$};
\draw[edge] (2) -- (1) node[midway, right] {$e_5$};
\draw[edge] (3) -- (2) node[midway, above] {$e_4$};
\draw[edge] (3) -- (1);
\draw[edge] (4) -- (2);


\node[fill=black!100, state, scale=0.10, vrtx/.style args = {#1/#2}{label=#1:#2}] (7) [vrtx=right/$v_7$] [right of = 15] {};

\node[fill=black!100, state, scale=0.10, vrtx/.style args = {#1/#2}{label=#1:#2}] (8) [vrtx=right/$v_8$] [ below right of = 7] {};

\node[fill=black!100, state, scale=0.10, vrtx/.style args = {#1/#2}{label=#1:#2}] (9) [vrtx=left/$v_9$] [ below left of = 7] {};

\node[fill=black!100, state, scale=0.10, vrtx/.style args = {#1/#2}{label=#1:#2}] (10) [vrtx=right/$v_{12}$] [ below right of = 8] {};

\node[fill=black!100, state, scale=0.10, vrtx/.style args = {#1/#2}{label=#1:#2}] (11) [vrtx=below/$v_{11}$] [ below left of = 8] {};

\phantom{\node[fill=black!100, state, scale=0.10, vrtx/.style args = {#1/#2}{label=#1:#2}] (12) [vrtx=left/$v_{10}$] [ below left of = 9] {};}

\draw[edge] (11) -- (9) node[midway, right] {$f_2$};
\draw[edge] (11) -- (8) node[midway, right] {$e_9$};
\draw[edge] (11) -- (10) node[midway, below] {$e_{11}$};
\draw[edge] (8) -- (7) node[midway, right] {$e_{7}$};
\draw[edge] (9) -- (8) node[midway, above] {$e_8$};
\draw[edge] (9) -- (7);
\draw[edge] (10) -- (8);
\draw[edge] (1) -- (7) node[midway, below] {$e_{6}$};

\end{tikzpicture}

\caption{The graph $\Gamma_2$.}
\label{example_2}
\end{figure}
\end{exam}


\section{Concluding remarks}\label{conrem}

We end the article with a non-example and some open questions. 

\begin{exam}
Consider the graph in \cref{stalling}, which is the $3$-fold join of two isolated vertices. 
The RAAG corresponding to this graph is the direct product of three copies of $F_2$ and the B-B group is the same group that appeared in the seminal paper of Stallings. 
It is a finitely presented group which is not of type $\mathbf{FP}_3$. 
The reader can easily verify that no matter which spanning tree is chosen there are at least two internal triangles that are unfavourable. 
In fact, any $k$-fold join of two isolated vertices has this property; so these graphs do not belong to the class $\mathcal{G}$. 
\begin{figure}[H]
    \centering
   \begin{tikzpicture}
    \draw (0,0) -- (2,0) -- (2,2) -- (0,2) -- (0,0);
    \draw (-2,1) -- (0,0) -- (4, 1);
    \draw (-2,1) -- (2,0) -- (4, 1);
    \draw (-2,1) -- (2,2) -- (4, 1);
    \draw (-2,1) -- (0,2) -- (4, 1);
\node [fill=black!100,circle,scale=0.3,draw] at (0,0) {};
\node [fill=black!100,circle,scale=0.3,draw] at (2,0) {};
\node [fill=black!100,circle,scale=0.3,draw] at (2,2) {};
\node [fill=black!100,circle,scale=0.3,draw] at (0,2) {};
\node [fill=black!100,circle,scale=0.3,draw] at (0,2) {};
\node [fill=black!100,circle,scale=0.3,draw] at (-2,1) {};
\node [fill=black!100,circle,scale=0.3,draw] at (4,1) {};
    \end{tikzpicture}
    \caption{A graph not in the class $\mathcal{G}$.}
    \label{stalling}
\end{figure}
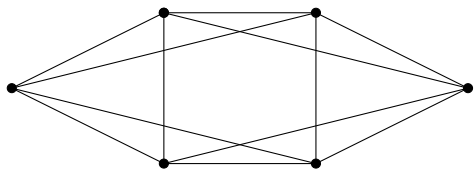
\end{exam}

\begin{que}
What is the complete characterization of the family $\mathcal{G}$? 
Moreover, if a Bestvina-Brady group $H_{\Gamma}$ decomposes as an iterated amalgamation (say, like the one specified in~\Cref{structure thm}) what can one say about the structure of $\Gamma$?
\end{que}

A group $G$ is called \textit{coherent} if each finitely generated subgroup of $G$ is finitely presented. 
Droms \cite[Theorem 1]{Droms} characterised coherent RAAGs in terms of the defining graph. 
A right-angled Artin group $A_\Gamma$ is \textit{coherent} if and only if $\Gamma$ does not have any cycle of length $n \geq 4$ as an induced subgraph (i.e., $\Gamma$ is chordal). 
Equivalently, $\Gamma$ has a separating clique and each component is either a clique or has a separating clique and so on. 

RAAGs corresponding to trees are obvious examples of coherent RAAGs. 
For example, $A_{P_4} \cong \ZZ^2 *_\ZZ \ZZ^2 *_\ZZ \ZZ^2$, the RAAG corresponding to the path on $4$ vertices is an important group. 
Kim and Koberda \cite[Theorem 7]{KK} proved that any $2$-dimensional coherent RAAG embeds in $A_{P_4}$.


Suppose $A_\Gamma$ is a coherent right-angled Artin group with connected defining graph $\Gamma$. 
Then the corresponding Bestvina-Brady group $H_\Gamma$ is also expressible as either an amalgamated product or a free product of free abelian (and infinite cyclic) groups. 
This follows directly from the definition of coherent RAAGs and \cite[Proposition 3.5]{Chang Thesis}. 
We recall (see \Cref{free gp BB}) that when $\Gamma$ is a finite tree, then $H_\Gamma$ is a free group of finite rank, i.e., a free product of finitely many copies of the infinite cyclic group. 
Most importantly, this decomposition also relies completely on the underlying graph structure. 
Note that the underlying graph of a coherent RAAG need not belong to the family $\mathcal{G}$.

We describe an example of a Bestvina-Brady group $H_\Gamma$ which decomposes as an amalgamated product over free-abelian groups but not covered by Theorem~\ref{structure thm}.
 
\begin{figure}[H]
\centering
\begin{tikzpicture}[shorten >=1pt,node distance=20cm,auto]

\node[fill=black!100, state, scale=0.10, vrtx/.style args = {#1/#2}{label=#1:#2}] (1) [vrtx=below/$v_1$] {};
\node[fill=black!100, state, scale=0.10, vrtx/.style args = {#1/#2}{label=#1:#2}] (2) [vrtx=below/$v_2$] [right of = 1] {};
\node[fill=black!100, state, scale=0.10, vrtx/.style args = {#1/#2}{label=#1:#2}] (3) [vrtx=right/$v_3$] [above of = 2] {};
\node[fill=black!100, state, scale=0.10, vrtx/.style args = {#1/#2}{label=#1:#2}] (4) [vrtx=left/$v_4$] [left of = 3] {};
\node[fill=black!100, state, scale=0.10, vrtx/.style args = {#1/#2}{label=#1:#2}] (5) [vrtx=left/$v_5$] [above left of = 4] {};
\node[fill=black!100, state, scale=0.10, vrtx/.style args = {#1/#2}{label=#1:#2}] (6) [vrtx=below/$v_6$] [below right of = 2] {};
\draw (1) -- (2) -- (3) -- (4)--(1);
\draw (1) -- (5) -- (4);
\draw (2) -- (6) -- (3);
\draw (1) -- (6);
\draw (5) -- (3);
\draw (1) -- (3);
\draw (2) -- (4);
\end{tikzpicture}

\caption{Defining graph of a coherent RAAG.}
    \label{fig:coherent RAAG}
\end{figure}
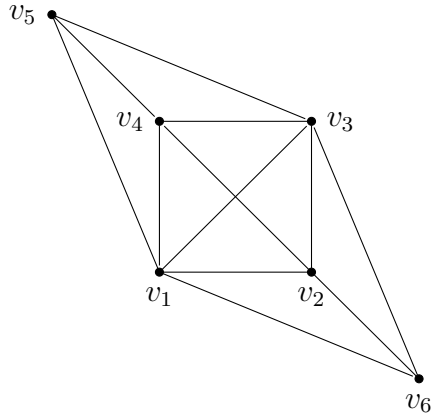

\begin{exam}
$A_\Gamma$ defined on the graph depicted in \Cref{fig:coherent RAAG} is coherenrt. 
It is clear that the induced subgraph $\Gamma_0$ on the set of vertices $\{v_1, v_2, v_3\}$ is a separating clique. 
Thus, $A_\Gamma= A_{\Gamma_1} *_{A_{\Gamma_0}} A_{\Gamma_2}$, where $\Gamma_1$ is the induced subgraph on the vertices $\{ v_1, v_2, v_3, v_4, v_5\}$ and $\Gamma_2$ is the induced subgraph on the vertices $\{ v_1, v_2, v_3, v_6\}$. 
Similarly, we can apply the same method of decomposition on $A_{\Gamma_1}$. 
As $\Gamma_1$ has the separating clique generated by the vertices $\{v_1, v_3, v_4\}$ and $\Gamma_2$ is itself a clique. Finally (see \cite[Corollary 3.12]{Chang-splitting}), we have $H_\Gamma \cong (\ZZ^3 *_{\ZZ^2} \ZZ^3)*_{\ZZ^2} \ZZ^3$.
\end{exam}

This stimulates the ensuing question:
\begin{que}
For which graphs can the corresponding B-B group be expressed as an iterated amalgamation of RAAGs? In other words,
what will be the possible characterization of the family of graphs $\mathcal{F}$ so that $H_\Gamma$ can be written as an iterated amalgamated product of RAAGs if and only if $\Gamma \in \mathcal{F}$?
\end{que}

\section*{Acknowledgements}

\noindent The authors grateful to the anonymous referee for insightful comments, suggesting the open problems and a detailed report, it helped refine the article.
The first named author is partially supported by a grant from the Infosys Foundation.
The second named author would like to express gratitude for the generous hospitality received from Universidad del Pa\'{\i}s Vasco and financial support through a Margarita Salas grant.



\begin{thebibliography}{99}

\bibitem{BRY} E. M. Barquinero, L. Ruffoni, K. Ye,
``Graphical splittings of Artin kernels'', \textit{J. Group Theory}  vol. 24, no. 4, (2021), pp. 711-735.

\bibitem{BB} M. Bestvina and N. Brady, ``Morse theory and finiteness properties of groups'', \textit{Invent. Math.} \textbf{129}(3) (1997), 445--470.

\bibitem{B} R. Bieri, ``Homological dimension of discrete groups'', \textit{Queen Mary College Mathematics Notes}, University of London (1976).






\bibitem{Chang Thesis} Y-C Chang, ``Dehn Functions of Bestvina-Brady Groups'', Ph.D. Thesis, 2019.

\bibitem{Chang Dehn functions} Y-C Chang, ``Identifying Dehn functions of Bestvina-Brady groups from their defining graphs'', \textit{Geometriae Dedicata}, \textbf{214}, 211--239, 2021.

\bibitem{Chang-splitting} Y-C Chang, ``Abelian splittings and JSJ-Decompositions of Bestvina–Brady groups'',  \textit{J. Group Theory} 2023.


\bibitem{Chang-Ruffoni} Y-C Chang, L. Ruffoni, ``A graphical description of the BNS-invariants of Bestvina-Brady groups and the RAAG recognition problem'', to appear in \textit{Groups, Geometry, and Dynamics}, 2023. 
\emph{arXiv:2212.06901}.


\bibitem{Charney} R. Charney, ``An introduction to right-angled Artin groups'', \textit{Geom. Dedicata} \textbf{125} (2007), 141--158.





\bibitem{Clay} M. Clay, ``When does a right-angled Artin group split over Z?'', \textit{Internat. J. Algebra Comput.}, \textbf{24}(6):815--825, 2014.

\bibitem{DL} W. Dicks and I. J. Leary, ``Presentations for subgroups of Artin groups'', \textit{Proc. Amer. Math. Soc.}
\textbf{127}(2) (1999), 343--348.

\bibitem{Droms} C. Droms, ``Graph groups, coherence, and three-manifolds'', \textit{Journal of Algebra}, 106:484--489, 1987.

\bibitem{GH} D. Groves, M. Hull, ``Abelian splittings of right-angled Artin groups'', In Hyperbolic geometry and geometric group theory, volume 73 of Adv. Stud. Pure Math., pages 159--165. Math. Soc. Japan, Tokyo, 2017.


\bibitem{KK} S. Kim and T. Koberda, ``Embedability between right-angled Artin groups'',
\textit{Geometry \& Topology}, 17:493–530, 2013.




\bibitem{PS} S. Papadima and A. Suciu, ``Algebraic invariants for Bestvina-Brady groups'', \textit{J. Lond. Math. Soc. (2)}, \textbf{76}(2) (2007), 273--292.




\bibitem{S} J. Stallings, ``A finitely presented group whose 3-dimensional integral homology is not finitely generated'', \textit{Amer. J. Math.} \textbf{85}(1963), 541--543.


\end{thebibliography}
\end{document}